\documentclass[11pt,a4paper]{amsart}

\usepackage[latin1]{inputenc}
\usepackage{amsmath,amsfonts,amssymb,setspace,amscd}
\usepackage{enumerate}
\usepackage[
      colorlinks=true,    
      urlcolor=blue,    
      menucolor=blue,    
      linkcolor=blue,    
      bookmarks=true,    
      bookmarksopen=true,    
      citecolor=blue,
      hyperfootnotes=false,    
      pdfpagemode=UseOutlines    
]{hyperref}
\usepackage[margin=2.5cm]{geometry}
\onehalfspace

\newtheorem{theorem}{Theorem}[section]
\newtheorem{lemma}[theorem]{Lemma}
\newtheorem{proposition}[theorem]{Proposition}

\theoremstyle{definition}
\newtheorem{definition}[theorem]{Definition}
\newtheorem*{remark}{Remark}

\numberwithin{equation}{section}
\numberwithin{theorem}{section}

\newcommand{\N}{\mathbb{N}}
\newcommand{\R}{\mathbb{R}}
\newcommand{\Z}{\mathbb{Z}}
\newcommand{\Q}{\mathbb{Q}}
\newcommand{\C}{\mathbb{C}}
\newcommand{\cA}{\mathcal{A}}
\newcommand{\cE}{\mathcal{E}}

\newcommand{\Qbar}{\overline{\Q}}

\newcommand{\cC}{\mathcal{C}}

\newcommand{\bo}[1]{\boldsymbol{#1}}

\def\lg{\left\lbrace}
\def\rg{\right\rbrace}

\author[F. Barroero]{Fabrizio Barroero}

\address{Scuola Normale Superiore\\
Piazza dei Cavalieri 7, 56126 Pisa\\
Italy}
\email{fbarroero@gmail.com}

\author[L. Capuano]{Laura Capuano}

\address{Scuola Normale Superiore\\
Piazza dei Cavalieri 7, 56126 Pisa\\
Italy}
\email{laura.capuano1987@gmail.com}

\title{Linear relations in families of powers of elliptic curves}
\subjclass[2010]{11G05 - 11G50 - 11U09 - 14K05 }   
\date{\today}

\thanks{The authors are supported by the ERC-Grant No. 267273.}

\begin{document}

\begin{abstract}
Motivated by recent work of Masser and Zannier on simultaneous torsion on the Legendre elliptic curve $E_\lambda$ of equation $Y^2=X(X-1)(X-\lambda)$, we prove that, given $n$ linearly independent points $P_1(\lambda), \dots,P_n(\lambda)$ on $E_\lambda$ with coordinates in $\overline{\mathbb{Q}(\lambda)}$, there are at most finitely many complex numbers $\lambda_0$ such that the points $P_1(\lambda_0), \dots,P_n(\lambda_0)$ satisfy two independent relations on $E_{\lambda_0}$. This is a special case of conjectures about Unlikely Intersections on families of abelian varieties.
\end{abstract}

\maketitle

\section{Introduction}

Let $n$ be an integer with $n\geq 2$ and let $E_\lambda$ denote the elliptic curve in the Legendre form defined by
\begin{equation}\label{legendre}
Y^2=X(X-1)(X-\lambda).
\end{equation}
In \cite{MasserZannier10} (see also \cite{MasserZannier08}), Masser and Zannier showed that there are at most finitely many complex numbers $\lambda_0 \neq 0,1$ such that the two points
$$
\left(2,\sqrt{2(2-\lambda_0)}\right), \    \    \   \  \left(3,\sqrt{6(3-\lambda_0)}\right), 
$$ 
both have finite order on the elliptic curve $E_{\lambda_0}$. Stoll \cite{Stoll} recently noted that there is actually no such $\lambda_0$. Later, in \cite{MZ12} Masser and Zannier proved that one can replace 2 and 3 with any two distinct complex numbers ($\neq 0,1$) or even choose distinct $X$-coordinates ($\neq \lambda$) defined over an algebraic closure of $\C(\lambda)$.

In his book \cite{Zannier}, Zannier asks if there are finitely many $\lambda_0 \in \C $ such that two independent relations between the points $\left(2,\sqrt{2(2-\lambda_0)}\right)$, $\left(3,\sqrt{6(3-\lambda_0)}\right)$ and $\left(5,\sqrt{20(5-\lambda_0)}\right)$ hold on $E_{\lambda_0}$.

In this article we prove that this question has a positive answer, as Zannier expected in view of very general conjectures. We actually prove a more general result, analogous to the one in \cite{MZ12} but, at the moment, we are only able to replace 2, 3 and 5 with any three pairwise distinct algebraic numbers, or choose $X$-coordinates defined over an algebraic closure of $\Q(\lambda)$, with the obvious exceptions 0, 1 and $\lambda$ because the corresponding points are identically 2-torsion. Moreover, our method allows us to deal with arbitrarily many points since we consider a curve $\cC\subseteq \mathbb{A}^{2n+1}$ with coordinate functions $(x_1,y_1,\dots ,x_n, y_n, \lambda)$, $\lambda$ non-constant, such that for every $j=1,\dots , n $, the points $P_j=(x_j,y_j)$ lie on the elliptic curve $E_\lambda$. As the point $\bo{c}$ varies on the curve $\cC$, the specialized points $P_j(\bo{c}) = (x_j(\bo{c}),y_j(\bo{c}))$ will be lying on the specialized elliptic curve $E_{\lambda(\bo{c})}$. We implicitly exclude the finitely many $\bo{c}$ with $\lambda(\bo{c})=0$ or 1, since in that case $E_{\lambda(\bo{c})}$ is not an elliptic curve.

We are now ready to state the main result of the article.

\begin{theorem} \label{mainthm}
Let $\cC\subseteq \mathbb{A}^{2n+1}$ be an irreducible curve defined over $\Qbar$ with coordinate functions $(x_1,y_1,\dots ,x_n, y_n, \lambda)$, 
$\lambda$ non-constant, such that, for every $j=1,\dots , n $, the points $P_j=(x_j,y_j)$ lie on $E_\lambda$ and there are no integers 
$a_1,\dots, a_n\in \Z$, not all zero, such that 
\begin{equation}\label{rel}
a_1P_1+\dots +a_nP_n= O,
\end{equation}
identically on $\cC$. Then there are at most finitely many $\bo{c}\in \cC$ such that the points $P_1(\bo{c}),\dots , P_n(\bo{c})$ satisfy two independent relations on $E_{\lambda(\bo{c})}$.
\end{theorem}

Note that the case $n=2$ is covered by the main proposition of \cite{MZ12} in the more general setting of a curve defined over $\C$.

Moreover, in \cite{RemVia} R\'emond and Viada proved an analogue of Theorem \ref{mainthm} for a power of a constant elliptic curve with CM, where one must allow the coefficients $a_1, \dots ,a_n$ in (\ref{rel}) to lie in the larger endomorphism ring. For the general case of powers of a constant elliptic curve, the result follows from works of Viada \cite{Viada2008} and Galateau \cite{Galateau2010}. If $n=2$ this is nothing but Raynaud's Theorem \cite{Rayn}, also known as the Manin-Mumford Conjecture.

We already mentioned the example of the three points with fixed abscissas 2, 3 and 5. It is easy to see that this will follow from Theorem \ref{mainthm} once we show that there is no identical relation between the three points on the generic curve $E_\lambda$. Indeed, the minimal fields of definition of these three points are disjoint quadratic extensions of $\Qbar(\lambda)$, and by conjugating one can see that the points would be identically torsion on $E_\lambda$. This is not possible, as it can be seen in different ways (see \cite{Zannier}, p.\,68). For instance, applying the Lutz-Nagell Theorem (\cite{Silv09}, Corollary 7.2), one can show that the point of abscissa 2 is not torsion on $E_6$.

One may ask if finiteness holds if we impose only one relation. This is not the case. Indeed, there are infinitely many $\lambda_0$ such that a point with fixed algebraic abscissa is torsion (see Notes to Chapter 3 in \cite{Zannier}). On the other hand, the values of $\lambda$ such that at least one relation holds are ``sparse'', as follows from work of Masser \cite{Masser89B}. Actually, a well-known theorem of Silverman \cite{Sil83} implies that the absolute Weil height of such values is bounded. A direct effective proof of this can be found in Masser's Appendix C of \cite{Zannier}. In particular, there are at most finitely many $\lambda_0$ yielding one relation in a given number field or of bounded degree over $\Q$.

Our proof follows the general strategy introduced by Pila and Zannier in \cite{PilaZannier} and used by Masser and Zannier in various articles \cite{MasserZannier08}, \cite{MasserZannier10} and \cite{MZ12}. In particular, we consider the elliptic logarithms $z_1, \dots , z_n$ of $P_1 , \dots , P_n$ and the equations
$$
z_j=u_jf+v_jg,
$$
for $j=1,\dots, n$, where $f$ and $g$ are suitably chosen basis elements of the period lattice of $E_\lambda$. If we consider the coefficients $u_j, v_j$ as functions of $\lambda$ and restrict them to a compact set, we obtain a subanalytic surface $S$ in $\R^{2n}$. The points of $\cC$ that yield two independent relations on the elliptic curve will correspond to points of $S$ lying on linear varieties defined by equations of some special form and with integer coefficients. In case $n=2$, one faces the simpler problem of counting rational points with bounded denominator in $S$. For this, a previous result of Pila \cite{Pila04} suffices together with the fact that the surface is ``sufficiently'' transcendental. In the general case we adapted ideas of Pila building on previous work of him \cite{Pila11} and used by the second author in her Ph.D. thesis \cite{Capuano} (see also \cite{CMPZ}), and obtained an upper bound of order $T^\epsilon$ for the number of points of $S$ lying on subspaces of the special form mentioned above and rational coefficients of absolute value at most $T$, provided $S$ does not contain a semialgebraic curve segment. Under the hypothesis that no identical relation holds on $\cC$, using a result of Bertrand \cite{Bertr}, we are able to show that there are no such semialgebraic curve segments.

Now, to conclude the proof, we use works of Masser \cite{Masser88}, \cite{Masser89} and David \cite{David97} and exploit the boundedness of the height to show that the number of points of $S$ considered above is of order at least $T^\delta$ for some $\delta>0$. Comparing the two estimates leads to an upper bound for $T$ and thus for the coefficients of the two relations, concluding the proof.

With methods similar to ours, in the works \cite{Capuano} and \cite{CMPZ}, mentioned above, the authors prove a toric analogue of Theorem \ref{mainthm}, giving an alternative proof of a result appeared in \cite{BMZ99} and later generalized by Maurin in \cite{Maurin} (see also \cite{BMZ08}).

In the paper, we will denote by $\gamma_1, \gamma_2, \ldots$ some positive constants. The indexes are reset at the end of each section.

\section{The Zilber-Pink Conjectures}

In this section we see how our Theorem relates to the so-called Zilber-Pink Conjectures on Unlikely Intersections.

First, let us examine the objects we are investigating from the point of view of dimensions. We consider our elliptic curve $E_\lambda$ as an elliptic scheme over $\mathbb{P}^1 \setminus \lg 0,1, \infty \rg$. Our ambient space is then the fiber power of $n$ copies of this elliptic scheme and has dimension $n+1$. Now, for any choice of linearly independent vectors $(a_1, \dots, a_n),(b_1, \dots, b_n) \in \Z^n$, imposing the two corresponding conditions yields an $(n-1)$-fold. Therefore, the intersection of a curve and an $(n-1)$-fold in a space of dimension $n+1$ is indeed unlikely to be non-empty and one expects finiteness for varying integer vectors. \\

Our result fits in the framework of very general conjectures formulated by Zilber \cite{Zilber} and Bombieri, Masser and Zannier \cite{BMZ07} in the toric case and by Pink \cite{Pink} in a more general setting, also known as the Zilber-Pink Conjectures.

In a series of papers \cite{MasserZannier10}, \cite{MZ12}, \cite{MZ14a} and \cite{MZ14b} Masser and Zannier proved a variant of Pink's conjecture in the case of a curve in an abelian surface scheme over $\Qbar$, and over $\C$ in the non-simple case. On the other hand, Pink's conjecture concerns families of semiabelian varieties. However, in 2011 Bertrand \cite{Bertrand11} found a counterexample to this,
for a suitable non-split extension of a CM elliptic constant family $E_0 \times B$ (over a curve $B$) by $\mathbb{G}_m$. This situation is rather ``special''; in fact, as it is shown in \cite{BMPZ}, the possible presence of the so called ``Ribet Sections'' is the only obstruction to the validity of the conjecture in the case of semiabelian surface schemes. \\ 

Now, let us see how our Theorem \ref{mainthm} implies a statement in the spirit of the conjectures mentioned above. In particular, we translate our result in the language of schemes, borrowing some terminology and results from a work of Habegger \cite{Hab}.

Let $S$ be an irreducible and non-singular quasi-projective curve defined over $\Qbar$ and let $\mathcal{E}\rightarrow S$ be an elliptic scheme over $S$, i.e. a group scheme whose fibers are elliptic curves. Let $n\geq 2 $. We define $\mathcal{A}$ to be the $n$-fold fibered power $\cE \times_S \dots \times_S \cE$ with the structural morphism $\pi:\cA\rightarrow S$. We suppose that $\cE$ is not isotrivial. In other words, $\mathcal{E}\rightarrow S$ cannot become a constant family after a finite \'etale base change.

A subgroup scheme $G$ of $\cA$ is a closed subvariety, possibly reducible, which contains the image of the zero section $S\rightarrow \cA$, is mapped to itself by the inversion morphism and such that the image of $G\times_S G$ under the addition morphism is in $G$.
A subgroup scheme $G$ is called flat if $\pi_{|_G} : G \rightarrow S$ is flat, i.e., all irreducible components of $G$ dominate the base curve $S$ (see \cite{Hart}, Proposition III 9.7).

We can now state the following theorem.

\begin{theorem}\label{thmscheme}
Let $\cA$ be as above and let $\cA^{\{2\}}$ be the union of its flat subgroup schemes of codimension at least 2. Let $\cC$ be a curve in $\cA$ defined over $\Qbar$ and suppose $\pi(\cC)$ dominates $S$. Then $\cC\cap \cA^{\{2\}}$ is contained in a finite union of flat subgroup schemes of positive codimension.
\end{theorem}

In order to prove that this theorem is a consequence of Theorem \ref{mainthm}, we need some notation and facts from \cite{Hab}.

For every $\bo{a}=(a_1, \dots , a_n)\in \Z^n $ we have a morphism $\bo{a}:\cA \rightarrow \cE$ defined by
$$
\bo{a}(P_1,\dots, P_n)=a_1 P_1 +\dots +a_n P_n.
$$
We identify the elements of $\Z^n$ with the morphisms they define. The fibered product $\alpha =\bo{a}_1 \times_S \dots \times_S \bo{a}_r$, for $\bo{a}_1 , \dots , \bo{a}_r \in \Z^n$ defines a morphism $\cA \rightarrow \mathcal{B}$ over $S$ where $\mathcal{B}$ is the $r$-fold fibered power of $\cE$. The kernel of $\alpha$, $\text{ker} \, \alpha$ indicates the fibered product of $\alpha: \cA \rightarrow \mathcal{B}$ with the zero section $S\rightarrow \mathcal{B}$. We consider it as a closed subscheme of $\cA$.

\begin{lemma}\label{lemrel}
Let $G$ be a  codimension $r$ flat subgroup scheme of $\cA$ with $1\leq r\leq n$. Then, there exist independent $\bo{a}_1,\dots , \bo{a}_r \in \Z^n$ such that $G\subseteq \textnormal{ker}(\bo{a}_1 \times_S \dots \times_S \bo{a}_r)$. Moreover, $\textnormal{ker}(\bo{a}_1 \times_S\dots \times_S \bo{a}_r)$ is a flat subgroup scheme of $\cA$ of codimension $r$.
\end{lemma}
\begin{proof}
This follows from Lemma 2.5 of \cite{Hab} and its proof.
\end{proof}

Consider now the Legendre family with equation \eqref{legendre}. As said before, this gives an example of an elliptic scheme, which we call $\cE_L$, over the modular curve $Y(2)=\mathbb{P}^1\setminus \{0,1,\infty\}$. We write $\cA_L$ for the $n$-fold fibered power of $\cE_L$.

\begin{lemma}[\cite{Hab}, Lemma 5.4] \label{lemdiag}
Let $\cA$ be as above. After possibly replacing $S$ by a Zariski open, non-empty subset there exists an irreducible, non-singular quasi-projective curve $S'$ defined over $\Qbar$ such that the following is a commutative diagram
$$
\begin{CD}
\cA @<{f}<< \cA' @>\text{e}>>  \cA_L \\
@V{\pi}VV @VVV @VV{\pi_L}V\\
S @<<{l}<  S' @>>{\lambda}> Y(2)
\end{CD}
$$
where $l$ is finite, $\lambda$ is quasi-finite, $\cA'$ is the abelian scheme $\cA\times_S S'$, $f$ is finite and flat and $e$ is quasi-finite and flat. Moreover, the restriction of $f$ and $e$ to any fiber of $\cA'\rightarrow S'$ is an isomorphism of abelian varieties.
\end{lemma}

We will also need the following technical lemma.

\begin{lemma} \label{lemhab}
If $G$ is a flat subgroup scheme of $\cA$ then $e\left(f^{-1}(G)\right)$ is a flat subgroup scheme of $\cA_L$ of the same dimension.
Moreover, let $X$ be a subvariety of $\cA$ dominating $S$ and not contained is a proper flat subgroup scheme of $\cA$, $X''$ an irreducible component of $f^{-1}(X)$ and $X'$ the Zariski closure of $e(X'')$ in $\cA_L$. Then $X'$ has the same dimension of $X$, dominates $Y(2)$ and is not contained in a proper flat subgroup scheme of $\cA_L$.
\end{lemma}

\begin{proof}
This follows from the proof of Lemma 5.5 of \cite{Hab}.
\end{proof}

\begin{proof}[Proof of Theorem \ref{thmscheme}]
First, we can assume that $\cC$ is not contained in a flat subgroup scheme of $\cA$ of positive codimension. Therefore, it is enough to prove that $\cC \cap \bigcup G$ is finite where the union is taken over all flat subgroup schemes of $\cA$ of codimension at least 2. 

Consider the Zariski closure $\cC'$ of $e(\cC'')$ for a component $\cC''$ of $f^{-1}\left(\cC \right) $. By Lemma \ref{lemhab}, $\cC'$ is a curve in $\cA_L$ dominating $Y(2)$ and not contained in a proper flat subgroup scheme.

Now, since $e$ is quasi-finite, if $e\left(f^{-1}\left(\cC\cap \cA^{\{2\}}\right)\right)$ is finite then $\cC\cap \cA^{\{2\}}$ is finite and by Lemma \ref{lemhab} we have
$$
e\left(f^{-1}\left(\cC\cap \cA^{\{2\}}\right)\right) \subseteq e\left(f^{-1}\left(\cC \right) \right) \cap \cA_L^{\{2\}}.
$$
Therefore, we can reduce to proving our claim for the Legendre family and for $\cC'$. 

By Lemma \ref{lemrel}, each flat subgroup scheme of codimension at least 2 of $\cA_L$ is contained in $\textnormal{ker}(\bo{a}_1 \times_{Y(2)}  \bo{a}_2)$ for some independent $\bo{a}_1, \bo{a}_2 \in \Z^n$. Therefore, it is enough to show that $\cC'\cap \bigcup \textnormal{ker}(\bo{a}_1 \times_{Y(2)}   \bo{a}_2)$ is finite, where the union is taken over all pairs of independent $\bo{a}_1, \bo{a}_2 \in \Z^n$. The claim follows applying our Theorem \ref{mainthm} since $\cC'$ is not contained in a proper flat subgroup scheme. 
\end{proof}

\section{O-minimal structures and a result of Pila}

In this section we introduce the notion of o-minimal structure, recall some definitions and properties we will need later and state a result of Pila from \cite{Pila11}. For the basic properties of o-minimal structures we refer to \cite{vandenDries1998} and \cite{DriesMiller}.

\begin{definition}
A \textit{structure} is a sequence $\mathcal{S}=\left( \mathcal{S}_N\right)$, $N\geq 1$, where each $\mathcal{S}_N$ is a collection of subsets of $\R^N$ such that, for each $N,M \geq 1$:
\begin{enumerate}
\item $\mathcal{S}_N$ is a boolean algebra (under the usual set-theoretic operations);
\item $\mathcal{S}_N$ contains every semialgebraic subset of $\R^N$;
\item if $A\in \mathcal{S}_N$ and $B\in \mathcal{S}_M$ then $A\times B \in \mathcal{S}_{N+M}$;
\item if $A \in \mathcal{S}_{N+M}$ then $\pi (A) \in \mathcal{S}_N$, where $\pi :\R^{N+M}\rightarrow \R^N$ is the projection onto the first $N$ coordinates.
\end{enumerate}
If $\mathcal{S}$ is a structure and, in addition,
\begin{enumerate}
\item[(5)] $\mathcal{S}_1$ consists of all finite union of open intervals and points
\end{enumerate}
then $\mathcal{S}$ is called an \textit{o-minimal structure}.
\end{definition}
Given a structure $\mathcal{S}$, we say that $S \subseteq \R^N$ is a \textit{definable} set if $S\in \mathcal{S}_N$. 

Let $U\subseteq \R^{N+M}$ and let $\pi_1$ and $\pi_2$ be the projection maps on the first $N$ and on the last $M$ coordinates, respectively. Now, for $t_0\in \pi_2(U)$, we set $U_{t_0}=\{ x\in \R^N: (x,t_0) \in U \}=~\pi_1 \left(\pi_2^{-1}(t_0)\right)$ and call $U$ a \emph{family} of subsets of $\R^N$, while $U_{t_0}$ is called the \emph{fiber} of $U$ above $t_0$. If $U$ is a definable set then we call it a \emph{definable family} and one can see that the fibers $U_{t_0}$ are definable sets too. Let $S\subseteq \R^N$ and $f:S\rightarrow \R^M$ be a function. We call $f$ a \emph{definable function} if its graph $\lg (x,y) \in S\times \R^{M}:y=f(x) \rg$ is a definable set. It is not hard to see that images and preimages of definable sets via definable functions are still definable.

There are many examples of o-minimal structures, see \cite{DriesMiller}. In this article we are interested in the structure of \textit{globally subanalytic sets}, usually denoted by $\R_{\text{an}}$. We are not going to pause on details about this structure because it is enough for us to know that if $D\subseteq \R^N$ is a compact definable set, $I$ is an open neighborhood of $D$ and $f:I\rightarrow \R^M$ is an analytic function then $f(D)$ is definable in $\R_{\text{an}}$.

We now fix an o-minimal structure $\mathcal{S}$. Many important properties of o-minimal structures follow from the \textit{cell decomposition Theorem} (\cite{DriesMiller}, 4.2). One of these is the fact that definable families have a uniform bound on the number of connected components of the fibers.

\begin{proposition}[\cite{DriesMiller}, 4.4] \label{unifbound}
Let $U$ be a definable family. There exists a positive integer $\gamma$ such that each fiber of $U$ has at most $\gamma$ connected components.
\end{proposition}

Now, let $S\subseteq\R^N$ be a non-empty definable set and let $e$ be a non-negative integer. The set of \emph{regular points} of dimension $e$, denoted by $\text{reg}_e(S)$, is the set of points $x\in S$ such that there is an open neighborhood $I$ of $x$ with $S\cap I$ a $C^1$ (embedded) submanifold of $\R^N$ of dimension $e$.
The \textit{dimension} of $S$ is the maximum $e$ such that $S$ has a regular point of dimension $e$. Note that, if $S$ has dimension $e$ then  $S\setminus \text{reg}_{e}(S)$ has dimension $\leq e-1$.

\begin{definition}
A \emph{definable block} of dimension $e$ in $\R^N$ is a connected definable set $B$ of dimension $e$ contained in some semialgebraic set $A$ of dimension  $e$ such that every point of $B$ is a regular point of dimension $e$ in $B$ and $A$. Dimension zero is allowed: a point is a definable block. Moreover, a \emph{definable block family} is a definable family whose non-empty fibers are all definable blocks.
\end{definition}

We now need to define the height of a rational point. The height used by Pila in \cite{Pila11} is not the usual projective Weil height, but a coordinatewise affine height. If $a/b$ is a rational number written in lowest terms, then $H(a/b)= \max (|a|,|b|)$ and, for an $N$-tuple $(\alpha_1, \dots , \alpha_N) \in \Q^N$, we set $H(\alpha_1, \dots , \alpha_N)= \max H(\alpha_i)$.
For a subset $Z$ of $\R^N$ and a positive real number $T$ we define
\begin{equation}\label{def}
Z(\Q,T)=\lg (\alpha_1, \dots , \alpha_N)  \in  Z\cap \Q^N: H(\alpha_1, \dots , \alpha_N) \leq T \rg.
\end{equation}

The following theorem is a special case of Theorem 3.6 of \cite{Pila11} (see also \cite{Pila09}). Here, if $f$ and $g$ are real functions of $T$, the notation $f(T) \ll_{Z,\epsilon}g(T)$ means that there exists a constant $\gamma$, depending on $Z$ and $\epsilon$, such that $f(T)\leq \gamma g(T)$ for $T$ large enough. 

\begin{theorem}[Pila, \cite{Pila11}] \label{ThmPila}
Let $Z\subseteq \R^N\times \R^M $ be a definable family, and $\epsilon >0$. There exist $J=J(Z,\epsilon)\in \N$ and a collection of definable block families $B^{(j)}\subseteq \R^N \times( \R^M \times \R^{M_j})$, for $j=1,\dots , J$, such that
\begin{enumerate}
\item  each point in each fiber of $B^{(j)}$ is regular of dimension $e_j$;
\item for each $(t,u) \in  \R^M \times \R^{M_j}$, the fiber $B^{(j)}_{(t,u)} \subseteq Z_t$;
\item for every $t \in  \pi_2(Z)$, $Z_t(\Q,T)$ is contained in the union of $\ll_{Z,\epsilon} T^\epsilon$ definable blocks, each a fiber of one of the $B^{(j)}$.
\end{enumerate}
\end{theorem}

\section{Points lying on rational linear varieties}

Let $n\geq 2$ be an integer and let $\ell_{1}, \dots ,\ell_{n},f,g$ be holomorphic functions on a connected neighborhood $I$ of some closed disc $D\subseteq \C$. 
Suppose that
\begin{equation}\label{tansc}
\ell_{1}, \dots ,\ell_{n} \text{ are algebraically independent over $\C(f,g)$ on $D$},
\end{equation} 
and $f(\lambda)$ and $g(\lambda)$ are $\R$-linearly independent for every $\lambda \in D$.

For some positive real $T$, denote by $D(T)$ the set of $\lambda \in D$ such that 
\begin{equation}\label{system2}
\lg
\begin{array}{c}
a_{1}\ell_{1}(\lambda)+\dots + a_{n} \ell_{n}(\lambda)=a_{n+1}f(\lambda)+ a_{n+2}g(\lambda)\\
b_{1}\ell_{1}(\lambda)+\dots + b_{n} \ell_{n}(\lambda)=b_{n+1}f(\lambda) + b_{n+2}g(\lambda)
\end{array}
\right.
\end{equation} 
for some linearly independent vectors $(a_{1},\dots , a_{n}),(b_{1},\dots , b_{n}) \in \left( \Z\cap [-T,T] \right)^{n}$
and some $a_{n+1}$, $a_{n+2}$, $b_{n+1}$, $b_{n+2}\in \Z$. 

The following proposition gives the desired upper bound mentioned in the introduction. We postpone its proof until the end of this section after developing some auxiliary tools.

\begin{proposition} \label{mainprop2}
Under the above hypotheses, for every $\epsilon>0$, $|D(T)|\ll_\epsilon T^{\epsilon}$.
\end{proposition}

Define
$$
\Delta=f \overline{ g }-\overline{f} g,
$$
which does not vanish on $D$, since $f(\lambda)$ and $g(\lambda)$ are $\R$-linearly independent for every $\lambda \in D$. Moreover, let
\begin{align*}
u_j=\frac{\ell_j \overline{g}-\overline{\ell_j} g}{\Delta},   \   \   \   v_j=-\frac{\ell_j \overline{ f}-\overline{\ell_j} f}{\Delta}.
\end{align*}
One can easily check that these are real-valued and, furthermore, we have
$$
\ell_j= u_jf+v_jg.
$$
If we view $D$ and $I$ as a subsets of $\R^2$, then $u_j$ and $v_j$ are real analytic functions on $I$.

Define
\begin{align*}
\Theta : D& \rightarrow \R^{2n}\\
 \lambda &\mapsto (u_{1}(\lambda),v_{1}(\lambda),\dots ,u_{n}(\lambda),v_{n}(\lambda))
\end{align*}
and set $S=\Theta(D)$. This is a definable set in $\R_{\text{an}}$.
In what follows, $(u_{1},v_{1},\dots ,u_{n},v_{n})$ will just indicate coordinates in $ \R^{2n}$. 

For $T>0$, we call $S(T)$ the set of points of $S$ of coordinates $(u_{1},v_{1},\dots ,u_{n},v_{n})$ such that there exist linearly independent vectors $(a_{1},\dots , a_{n+2}),(b_{1},\dots , b_{n+2}) \in  \Q^{n+2}$ of height at most $T$ with
\begin{equation}\label{system}
\lg
\begin{array}{c}
a_{1}u_{1}+\dots + a_{n} u_{n}=a_{n+1}\\
a_{1}v_{1}+\dots + a_{n} v_{n}=a_{n+2}\\
b_{1}u_{1}+\dots + b_{n} u_{n}=b_{n+1}\\
b_{1}v_{1}+\dots + b_{n} v_{n}=b_{n+2}.
\end{array}
\right.
\end{equation} 
We will need the following lemma.
\begin{lemma}\label{lemcurve}
For every choice of $a_{1},\dots , a_{n+2},b_{1},\dots , b_{n+2} \in \R$ (not only rationals), not all zero, the subset of $S$ for which (\ref{system}) holds is finite.
\end{lemma}

\begin{proof}
By contradiction suppose that the subset of $S$ of points satisfying (\ref{system}) for some choice of coefficients is infinite. This 
would imply that there exists an infinite set $E\subseteq D$ on which, for instance, if at least one of the $a_j\neq 0$, for every $\lambda \in E$,
$$
a_{1}\ell_{1}(\lambda)+\dots + a_{n} \ell_{n}(\lambda)=a_{n+1}f(\lambda)+ a_{n+2} g(\lambda).
$$
Since this relation holds on a set with an accumulation point, it must hold on the whole $D$ (see Ch. III, Theorem 1.2 (ii) of \cite{LangCompl}), contradicting the hypothesis (\ref{tansc}).
\end{proof}

The following proposition is the main tool to prove Proposition \ref{mainprop2}.

\begin{proposition}\label{mainprop}
For every $\epsilon >0$ we have
$$
|S(T)|\ll_{\epsilon} T^{\epsilon}.
$$
\end{proposition}

\begin{proof}
We are counting points of $S$ that lie on linear varieties of $\R^{2n}$ defined by systems of the form \eqref{system}.

Let us consider the set $W \subset \R^{4n+4}$ defined as
\begin{multline*}
W=\lg (u_{1},v_{1},\dots ,u_{n},v_{n},a_{1},\dots , a_{n+2},b_{1},\dots , b_{n+2} )\in S \times \R^{2n+4}:\right.\\ \left. 
\text{ with (\ref{system}) and $(a_{1},\dots , a_{n+2})$ and $(b_{1},\dots , b_{n+2})$ are linearly independent} \vphantom{\R^{2n+4}\times S} \rg,
\end{multline*}
which is a definable set.
Denote by $\pi_1$ the projection on $S$ and by $\pi_2$ the projection on the last $2n+4$ coordinates.
Given a point $L$ of $\pi_2(W)$, we write $\tau(L)$ for the set of points of $S$ that lie on the affine subspace corresponding to $L$, i.e., 
$\tau(L)=\pi_1\left(\pi_2^{-1}(L)\right)$. In other words, if we consider $W$ as a family of subsets of $\R^{2n}$, $\tau(L)$ is just the fiber $W_L$. This is a definable subset of $S$ and we have that it must be zero-dimensional by Lemma \ref{lemcurve}. 
By Proposition \ref{unifbound}, there exists a positive integer $\gamma_1$ such that $|\tau(L)|\leq \gamma_1$, for every $L \in \pi_2(W)$.  Moreover, for $V\subseteq \pi_2(W)$, we write $\tau(V)$ for $\pi_1\left(\pi_2^{-1}(V)\right)$.

Now, let us call $\widehat{W}=\pi_2(W) \subseteq \R^{2n+4}$. Recall the definition in \eqref{def} and note that $S(T)\subseteq \tau \left( \widehat{W}(\Q, T)\right)$. By Theorem \ref{ThmPila}, there is a finite number of definable block families such that, for every $\epsilon_1$, $\widehat{W}(\Q, T)$ 
is contained in the union of $\ll_{W, \epsilon_1} T^{\epsilon_1}$ definable blocks, each a fiber of one of these families. We have that $S(T) \subseteq \bigcup_B \tau (B)$ where the union is taken over the $\ll_{W, \epsilon_1 } T^{\epsilon_1}$ definable blocks mentioned above.

Let us fix a definable block family $U$ with fibers $U_t \subseteq \widehat{W}$. We claim that, for every $\epsilon_2$, each fiber $U_t$ of $U$ gives rise to 
$\ll_{U, \epsilon_2} T^{\epsilon_2}$ points on $S(T)$, i.e., that $|\tau(U_t)\cap S(T)|\ll_{U, \epsilon_2} T^{\epsilon_2}$ for every fiber $U_t$. Once we proved this, the claim of the proposition follows easily after fixing $\epsilon_1$ and $\epsilon_2$ with $\epsilon_1\epsilon_2=\epsilon$, e.g., $\epsilon_1=\epsilon_2=\sqrt{\epsilon}$.\\

We proceed by induction on the dimension $e$ of the fibers of $U$.
By Lemma \ref{lemcurve}, the claim is true for $e=0$. 

Suppose now $e>0$. We denote by $B_\eta(L)$ the Euclidean ball centered in $L$ of radius $\eta$, and define, for $m=1,\dots ,\gamma_1$,
$$
V^{(m)}=\lg (L,t) \in U: \exists \eta>0, \, \exists A_1, \dots, A_m \in S: \forall L'\in B_\eta(L) \cap U_t, \, \tau(L')=\lg A_1, \dots, A_m \rg \rg.
$$
These are definable families and so is $V:=\bigcup_{m=1}^{\gamma_1} V^{(m)}$ as it is a finite union of definable sets. 
Hence, by Proposition \ref{unifbound}, there exists $\gamma_2$ such that all fibers $V_t$ have at most $\gamma_2$ connected components. 
It is clear that, for each $L$ in the same connected component, $\tau(L)$ consists of the same set of not more than $\gamma_1$ points;
therefore, each fiber $V_t$ of $V$ gives rise to at most $\gamma_1\gamma_2$ points of $S(T)$, i.e., $|S(T)\cap \tau (V_t)|\leq \gamma_1 \gamma_2$.

Now we want to prove that all the fibers of $Z=U \setminus V$ have dimension $<e$. Suppose not and let $L$ be an $e$-regular point of a fiber $Z_t$. Then, 
after fixing a ball $B_\eta(L)$ with connected $B_\eta(L)\cap U_t \subseteq Z_t$, we set $\{A_1, \dots , A_m \}=\bigcap_{L' \in B_\eta(L)\cap Z_t}\tau(L')$, 
i.e., the set of points of $S$ that lie on all subspaces in $B_\eta(L)\cap Z_t$. By definition of $Z$, $\tau^{-1}\left(\{A_1, \dots , A_m \}\right)\cap B_\eta(L)\cap Z_t$ must be of dimension $<e$, therefore there exist $L_0 \in B_\eta(L)\cap Z_t$ 
and $\eta_0$ such that for every $L'\in B_{\eta_0}(L_0)\cap Z_t$, $\tau(L') \supsetneq \{A_1, \dots , A_m \}$. Thus, we can define 
$f:B_{\eta_0}(L_0)\cap Z_t \rightarrow S$ that associates to $L'$ a point in $ \tau(L')$ different from $A_1, \dots , A_m$. 
This is a definable function and, taking $\eta_0$ small enough (and possibly choosing a different $L_0$), we can also suppose that it is differentiable (\cite{DriesMiller}, C.2 Lemma). 

Now, assume the derivative of $f$ is zero in all directions. Then, $f$ is constant and there exists a point $A_{m+1} \in \tau(L') $ for all $L' \in B_{\eta_0}(L_0)\cap Z_t$. 
We repeat this procedure of finding a point, a ball and a function like above and continue until this function has non-zero derivative in 
some direction. This procedure must stop because otherwise we would have a point $L'$ with $|\tau(L')|>\gamma_1$, contradicting the above considerations. 

We can therefore suppose that there are $L_0 \in B_\eta(L)\cap Z_t$ and $\eta_0$ such that $f$ is differentiable on $B_{\eta_0}(L_0)\cap Z_t$ and with 
non zero-derivative in some direction. Now, recall that $L_0$ is an $e$-regular point of $U_t$ and, by definition of definable block, of a semialgebraic set that contains it. Therefore, $B_{\eta_0}(L_0)\cap U_t=B_{\eta_0}(L_0)\cap Z_t$ is semialgebraic. Thus, if we intersect it with a suitable linear variety, we get an algebraic 
curve segment $C$ in $B_{\eta_0}(L_0)\cap Z_t$, passing through $L_0$ in the direction for which the derivative of $f$ is non-zero. The function $f$ is non-constant on $C$. 
Consider $C'= f(C)\times C$. By definition of $f$, $C'$ is a real-analytic curve segment in $W$.
Moreover, let us define $D'=\Theta^{-1}(f(C))$. As $f$ is not constant on $C$, $D'$ is an infinite subset of $D$. 

Now, on $D'$ the coordinate functions $a_{1},\dots , a_{n+2},b_{1},\dots , b_{n+2}$ satisfy $2n+3$ independent algebraic relations with 
coefficients in $\C$ and, combining the relations of (\ref{system}), we have also
\begin{equation*}
\lg
\begin{array}{c}
a_{1}\ell_{1}+\dots + a_{n} \ell_{n}=a_{n+1}f+ a_{n+2}g\\
b_{1}\ell_{1}+\dots + b_{n} \ell_{n}=b_{n+1}f + b_{n+2}g.
\end{array}
\right.
\end{equation*} 
Each of these two is independent of the previous $2n+3$ relations and they are independent of each other because we imposed $(a_{1},\dots , a_{n+2})$ and $(b_{1},\dots , b_{n+2})$ to be linearly independent.
Therefore, as the $3n+4$ functions $a_{1},\dots , a_{n+2},b_{1},\dots , b_{n+2},\ell_{1}, \dots ,\ell_{n}$ satisfy $2n+5$ independent algebraic relations with coefficients in $\C[f,g]$ on the infinite set $D'$, they continue to do so on $I$.
Therefore, if $F:=\C(f,g)$,
$$
\text{tr.deg}_F F(\ell_{1}, \dots ,\ell_n)< n.
$$
This contradicts the hypothesis (\ref{tansc}). \\

We have just proved that there cannot be any $e$-regular point on any fiber of $Z$. We apply Pila's result (Theorem \ref{ThmPila}) again on $Z$.
There is a finite number of definable block families such that, for each $\epsilon_3$ and for each fiber $Z_t$, $Z_t(\Q, T)$ is contained in the union of $\ll_{Z, \epsilon_3} T^{\epsilon_3}$ definable blocks, each a fiber of one of these families. The fibers of these families must have dimension $<e$, therefore our inductive hypothesis implies that if $U'$ is one of them, then, for every $\epsilon_4$, $|\tau(U'_{t'}) \cap S(T)|\ll_{U' , \epsilon_4} T^{\epsilon_4}$, for every fiber $U'_{t'}$ of $U'$. This means that, after choosing $\epsilon_3=\epsilon_4=\sqrt{\epsilon_2}$, for each fiber $Z_t$, we have $|\tau(Z_t)\cap S(T)|\ll_{Z, \epsilon_2} T^{\epsilon_2}$. 
Now recall that we had $U_t=V_t\cup Z_t$ and that $V_t$ gives rise to at most $\gamma_1\gamma_2$ points of $S(T)$. 
This proves our claim and the proposition.
\end{proof}

\begin{remark}
We would like to point out that this last proposition can be deduced from recent work of Habegger and Pila, in particular Corollary 7.2 of \cite{HabPila14}.
\end{remark}

We are now ready to prove Proposition \ref{mainprop2}.

\begin{proof}[Proof of Proposition \ref{mainprop2}]
Since $f$ and $g$ are linearly independent, if $\lambda \in D$ satisfies (\ref{system2}) then (\ref{system}) holds for $\Theta(\lambda)$. Now, since $D$ is a compact subset of $\R^2$, each $\ell_{j}(D)$ is bounded and therefore, if $\ell_{1}(\lambda), \dots ,\ell_{n}(\lambda),$ 
$f(\lambda),g(\lambda)$ satisfy (\ref{system2}), then $|a_{n+1}|, |a_{n+2}|,|b_{n+1}|, |b_{n+2}|$ are bounded in terms of 
$|a_{1}|,\dots , |a_{n}|,$ $|b_{1}|,\dots, |b_{n}|$ and thus of $T$. Therefore, $\Theta(\lambda) \in S(\gamma_3 T)$ for some $\gamma_3$ independent of $T$.
Now, using Proposition \ref{unifbound} and Lemma \ref{lemcurve}, we see that there exists a $\gamma_4$ such that, for any choice of $a_1,\dots ,a_{n+2}, b_1 ,\dots ,b_{n+2}$, there are at most 
$\gamma_4$ elements $\lambda$ in $D$ such that $\ell_{1}(\lambda), \dots ,\ell_{n}(\lambda)$, $f(\lambda),g(\lambda)$ satisfy (\ref{system2}). Thus $|D(T)|\ll |S(\gamma_3 T)|$ and the claim follows from Proposition \ref{mainprop}.
\end{proof}

\section{Periods and elliptic logarithms}\label{periods}

In this section we introduce the functions which we will apply Proposition \ref{mainprop2} to. We follow \cite{MZ12}.

It is well known that there is an analytic isomorphism between $E_\lambda(\C)$ and $\C/L_\lambda$, where $L_\lambda$ is a rank 2 lattice in $\C$. Consider the hypergeometric function
$$ F(t)=F\left(\frac{1}{2},\frac{1}{2},1;t\right)=\sum_{m=0}^{\infty} \frac{(2m)!^2}{2^{4m}m!^4}t^m, $$
and let
\begin{equation} \label{f,g}
f(t)=\pi F(t), \quad \mbox{and} \quad g(t)=\pi i F(1-t).
\end{equation}
Moreover, we define
$$
\Lambda=\lg t \in \C: |t|<1 , |1-t|<1 \rg.
$$
The functions $f$ and $g$ are well-defined and analytic in $\Lambda$, as functions of $t$. Moreover, they are well-defined as 
functions of $\bo{c}$ in $\lambda^{-1}(\Lambda)\subset \cC(\C)$.

By (6.1) Theorem, p.\,179, of \cite{Huse}, $f(\lambda)$ and $g(\lambda)$ are basis elements of the period lattice $L_\lambda$ of $E_\lambda$ with respect to
$\frac{dX}{2Y}$. Therefore, 
if $\exp_\lambda$ is the associated exponential map from $\C$ to $E_\lambda(\C)$, we have
$$
\exp_\lambda(f(\lambda))=\exp_\lambda(g(\lambda))=O,
$$
where $O$ denotes the origin in $E_{\lambda}$.
Let $P_j=(x_j,y_j)$, where $x_j,y_j$ are coordinate functions in $\C(\cC)$. We can suppose that, for every $j$, 
$x_j\neq 0,1,\lambda$ identically, otherwise the corresponding $P_j$ would be identically 2-torsion, contradicting the hypothesis of Theorem \ref{mainthm}.

Now, we want to define suitable functions $z_j(\bo{c})$ such that 
$\exp_{\lambda(\bo{c})}(z_j(\bo{c}))=P_j(\bo{c})$; in other words, we want $z_j$ to be the elliptic logarithm of $P_j$. 

Let $\widehat{\cC}$ be the subset of points $\bo{c}  \in \cC$ such that $\lambda(\bo{c}),x_j(\bo{c})\neq 0,1,\infty$, $ x_j(\bo{c})\neq \lambda (\bo{c}) $ for every $j=1,\dots, n$ and $\bo{c}$ is not a singular point or a point on which the differential of $\lambda$ vanishes.

Note that, in this way, we excluded finitely many $\bo{c}\in \cC$, and these are algebraic points of $\cC$. Moreover, on $\widehat{\cC}$, the coordinate function $\lambda$ has everywhere a local inverse.

We now follow the construction of \cite{MZ12}, p.\,459. Fix a point $\bo c_{*}\in \widehat{\cC}$ and choose a path in the $x_j$-plane from $x_j(\bo c_{*})$ to $\infty$ and not passing through 
$0,1$ and $\lambda(\bo c_{*})$. We also fix a determination of $Y=\sqrt{X(X-1)(X-\lambda(\bo c_*))}$ equal to $y_j(\bo c_{*})$ at $X=x_j(\bo c_{*})$. Therefore, the 
path corresponds to a path on the elliptic curve $E_{\lambda(\bo c_*)}$ from the point $P_j(\bo c_*)$ to the origin $O$. Hence we can define $z_j(\bo c_{*})$ as
the integral
$$
z_j(\bo{c}_{*})=\int_{x_j(\bo c_{*})}^{\infty} \frac{dX}{2\sqrt{X(X-1)(X-\lambda(\bo c_{*}))}}.
$$
We can extend it to a $\bo{c}$ close to $\bo c_{*}$ by
$$z_j(\bo{c}) =\int_{x_j(\bo c_{*})}^{\infty}\frac{dX}{2\sqrt{X(X-1)(X-\lambda(\bo c)})}+ \int_{x_j(\bo c_{*})}^{x_j(\bo c)}\frac{dX}{2\sqrt{X(X-1)(X-\lambda(\bo c)})} .$$ In fact, in the first integral on the right we use the same path fixed 
before and the integrand is determined by continuity from the previously chosen determination of $Y$. Hence, this term is an analytic function in 
$\lambda(\bo c)$. For the second term, we can take any local path from ${x_j(\bo c_{*})}$ to ${x_j(\bo c)}$. We can extend the integrand as a 
double power series in $\lambda(\bo c)-\lambda(\bo c_{*})$ and in $X-x_j(\bo c_{*})$; the result will be a double power series in 
$\lambda(\bo c)-\lambda(\bo c_{*})$ and $x_j(\bo c)-x_j(\bo c_{*})$. Notice that we have at any rate $\exp_{\lambda(\bo c)}(z_j(\bo{c}))=P_j(\bo{c})$ for every $j=1, \ldots, n$.

In this way, fixing a $\bo{c}_* \in \lambda^{-1}(\Lambda)\cap \widehat{\cC}$, the functions $z_1, \dots ,z_n$ are well defined 
on a small neighborhood $N_{*}$ on $\cC$. Moreover, if we take $N_{*}$ small enough, we can see them as analytic functions of $\lambda$ on $\lambda^{-1}(N_*)$. \\

We will need the following transcendence result.

\begin{lemma}\label{lemtransc1}
The functions $z_1, \dots, z_n$ are algebraically independent over $\C(f,g)$ on $N_*$.
\end{lemma}

\begin{proof}
The $z_1, \dots, z_n, f,g$ are analytic functions of $\lambda$, linearly independent over $\Z$. Indeed, a relation $a_1z_1+ \dots + a_n z_n = a_{n+1}f+a_{n+2}g$, with integer coefficients, would map via $\exp_\lambda$ to a relation of the form \eqref{rel} on $N_*$, and therefore on the whole $\cC$, which cannot hold by the hypothesis of the Theorem. Moreover, if $\wp_\lambda$ is the Weierstrass $\wp$-function associated to $L_\lambda$, the $\wp_\lambda(z_i)$ are algebraic functions of $\lambda$ because $\wp_\lambda(z_j)=x_j - \frac{1}{3}(\lambda +1 )$ (see (3.8) on p.\,7 of \cite{MasserZannier08}). Therefore, the hypotheses of Th\'eor\`eme 5 on p.\,136 of \cite{Bertr} are satisfied and we can apply it to get the claim. 
\end{proof}

We would like now to extend our functions $f,g,z_1, \ldots, z_n$ on $\widehat{\cC}$.

If $\bo{c}\in \widehat{\cC}$, one can continue $f$ and $g$ to a neighborhood $N_{\bo{c}}$ of $\bo{c}$.
In fact, if we choose $\bo{c}\in \widehat{\cC}$ and a path from $\bo{c}_*$ to $\bo{c}$ lying in $\widehat{\cC}$, we can easily continue $f$ and
$g$ along the path using (\ref{f,g}). \\

To continue $z_j$ from a point $\bo c_*$ to a $\bo c$ in $\widehat{\cC}$, it is sufficient to verify that, if $N_1$ and $N_2$ are two open small 
subsets in $\widehat{\cC}$, with $N_1 \cap N_2$ connected, and $z_j$ has analytic definitions $z_j'$ on $N_1$ and $z_j''$
on $N_2$, then it has an analytic definition on the union $N_1 \cup N_2$. 
But we saw that $\exp_\lambda(z_j)=P_j$ for every $j=1, \ldots, n$, hence on $N_1 \cap N_2$ we have $\exp_\lambda(z_j')=\exp_\lambda(z_j'')$.
This means that there exist rational integers $u,v$ with $z_j''=z_j'+uf+vg$ on this intersection, and they must be constant there. Hence it is enough 
to change $z_j''$ to $z_j''-uf-vg $ on $N_2$. 

Using the same path, it is now clear that we can continue the function $(f,g,z_1, \ldots, z_n)$ from a small neighborhood of $\bo c_*$ to a small 
neighborhood $N_{\bo c}\subseteq \widehat{\cC}$ of $\bo c$, and the obtained function $(f^{\bo{c}},g^{\bo{c}},z_1^{\bo{c}}, \dots ,z_n^{\bo{c}})$ is
analytic on $N_{\bo c}$. Moreover, the functions preserve the algebraic independence, as the following lemma shows.

\begin{lemma}\label{lemtransc}
The functions $z_1^{\bo{c}}, \dots, z_n^{\bo{c}}$ are algebraically independent over $\C(f^{\bo{c}},g^{\bo{c}})$ on $N_{\bo{c}}$. 
\end{lemma}

\begin{proof}
Any algebraic relation can be continued to a neighborhood $N_*$ of some $\bo{c}_* \in \lambda^{-1}(\Lambda)$, contradicting Lemma \ref{lemtransc1}.
\end{proof}

Furthermore, the lattice $L_{\lambda}$ is still generated by $f^{\bo{c}} $ and $g^{\bo{c}} $ on $N_{\bo{c}}$, see Lemma 6.1 of \cite{MZ12} or Lemma 4.1 of \cite{MasserZannier10}.

Now fix $\bo{c} \in \cC$ and $N_{\bo c}\subseteq \widehat{\cC}$. Since we are avoiding 
singular points and points on which the differential of $\lambda$ vanishes, $\lambda$ 
gives an analytic isomorphism $\lambda:N_{\bo{c}}\rightarrow \lambda(N_{\bo{c}})$. Then, we can view $z_1^{\bo{c}}, \dots, z_n^{\bo{c}},f^{\bo{c}},g^{\bo{c}}$ as analytic functions on $\lambda(N_{\bo{c}})$.

\section{Linear relations on a fixed curve}

In this section we prove a general fact about linear relations on elliptic curves. 

For a point $(\alpha_1, \dots ,\alpha_N) \in \overline{\Q}^N$, the absolute logarithmic Weil height $h(\alpha_1, \dots ,\alpha_N)$ is defined by
$$
h(\alpha_1, \dots ,\alpha_N)= \frac{1}{[\Q(\alpha_1, \dots ,\alpha_N):\Q]} \sum_v \log \max \{1,| \alpha_1|_v, \dots ,|\alpha_N|_v\},
$$
where $v$ runs over a suitably normalized set of valuations of $\Q(\alpha_1, \dots ,\alpha_N)$.

Let $\alpha$ be an algebraic number and consider the Legendre curve $E=E_\alpha$ defined by the equation
$ Y^2=X(X-1)(X-\alpha)$. Moreover, let $P_1, \dots , P_n$ be linearly dependent points on $E$, defined over some finite extension $K$ of $\Q(\alpha)$ of degree $\kappa=[K:\Q]$. Suppose that $P_1, \dots , P_n$ have N\'eron-Tate height $\widehat{h}$ at most $q$
(for the definition of N\'eron-Tate height, see for example p.\,255 of \cite{Masser88}). In case the $P_1, \dots , P_n$ are all torsion, i.e., $\widehat{h}(P_j)=0$ for all $j$, we set $q=1$.
We define
$$
L(P_1, \dots , P_n)=\{ (a_1, \dots , a_n) \in \Z^n: a_1P_1+\dots +a_nP_n= O \}.
$$
This is a sublattice of $\Z^n$ of some positive rank $r$. We want to show that $L(P_1, \dots , P_n)$ has a set of generators with small max norm 
$|\bo{a}|=\max \{|a_1|, \dots ,|a_n|\}$.

\begin{lemma}\label{lemgenerators}
Under the above hypotheses, there are generators $\bo{a}_1, \dots , \bo{a}_r$ of $L(P_1, \dots , P_n)$ with
$$
|\bo{a}_i|\leq \gamma_1  \kappa^{\gamma_{2}} (h(\alpha)+1)^{2n} q^{\frac{1}{2}(n-1)},
$$
for some positive constants $\gamma_1, \gamma_2$ depending only on $n$.
\end{lemma}

\begin{proof}
The Weierstrass form $\widetilde{E}=\widetilde{E}_\alpha$ of $E=E_\alpha$ has equation
$$
\widetilde{Y}^2=4\widetilde{X}^3-g_2 \widetilde{X}-g_3,
$$
where $g_2=\frac{4}{3} (\alpha^2 - \alpha+1)$ and $g_3=\frac{4}{27}(\alpha-2)(\alpha+1)(2\alpha-1)$ (see (3.7) of \cite{MasserZannier10}). The isomorphism $\phi$  from $E$ to 
$\widetilde{E}$ is given by
$$
\widetilde{X}=X-\frac{1}{3}(\alpha+1), \ \ \widetilde{Y}=2Y.
$$

Now, $\widetilde{E}$ is clearly defined over $\Q(\alpha)$ and any linear relation $a_1P_1+ \dots + a_nP_n=O$ on $E$ carries on to $\widetilde{E}$ and 
vice versa. Moreover, the $Q_i=\phi(P_i)$ will have coordinates in $K$ and the same N\'eron-Tate height of the $P_i$, hence also bounded by $q$.

First, suppose that at least one of the points has infinite order. By Theorem E of \cite{Masser88}, if $Q_1, \dots, Q_n$ are linearly dependent points on $\widetilde{E}(K)$, of N\'eron-Tate height at most $q \geq \eta$, 
then $L(Q_1, \dots, Q_n)$ is generated by vectors with max norm at most 
$$
n^{n-1} \omega \left(\frac{q}{\eta}\right)^{\frac{1}{2}(n-1)},
$$
where $\omega=\left|\widetilde{E}_{\footnotesize{\mbox{tors}}}(K)\right|$ and $\eta=\inf \widehat{h}(P)$, for $P\in \widetilde{E}(K)\setminus \widetilde{E}_{\footnotesize{\mbox{tors}}}(K)$.
We need to bound $\omega$ and $\eta$. The constants $\gamma_3, \dots, \gamma_9$ are absolute constants.

For the first we use a result of David \cite{David97}.
By Th\'eor\`eme 1.2 (i) of \cite{David97}, choosing any archimedean $v$ and noting that, by David's definition, 
$h_v\left(\widetilde{E}\right)\geq \frac{\sqrt{3}}{2}$, one has
$$
\omega \leq \gamma_3 (\kappa h+\kappa \log \kappa),
$$
where $h=\max \lg 1,h\left(j_{\widetilde{E}}\right) \rg$. Now, $j_{\widetilde{E}}=2^8\frac{(\alpha^2-\alpha+1)^3}{\alpha^2(\alpha-1)^2}$ 
(see for instance \cite{Huse}, p.\,83). Therefore, $h\leq \gamma_4 (h(\alpha)+1)$ and we have
\begin{equation} \label{tor}
\omega \leq \gamma_5 (h(\alpha)+1) \kappa^2.
\end{equation}
For the lower bound on $\eta$, we use a result of Masser (Corollary 1 of \cite{Masser89}). In Masser's bound a constant depending on $\kappa$ appears in the denominator. 
However, going through the proof one can see that this constant is polynomial in $\kappa$, as noted by David on p.\,109 of \cite{David97}. Therefore,
$$
\eta\geq \frac{\gamma_{6} }{w\kappa^{\gamma_7 + 3}(w+\log \kappa)^2}\geq \gamma_{8} \kappa^{-(\gamma_{7}+5)}w^{-3},
$$
where $w=\max \{1, h(g_2) ,h(g_3)  \}$. As $g_2$ and $g_3$ are polynomials in $\alpha$, $w\leq \gamma_9 (h(\alpha)+1)$. 
Consequently, $L(Q_1, \dots , Q_n)$ will have generators of norms at most
$$
\gamma_1 \kappa^{\gamma_{2}} (h(\alpha)+1)^{2n} q^{\frac{1}{2}(n-1)},
$$
with $\gamma_1, \gamma_2$ depending only on $n$.

In case all the points are torsion points, it is clear that one can take $|\bo{a}_i|\leq \omega$ and use \eqref{tor}.
\end{proof}

\section{Bounded height}

In this section we see that the height of the points on the curve $\cC$ for which there is at least one dependence relation is bounded and a few consequences of this fact.

Let $k$ be a number field over which $\cC$ is defined. Suppose also the finitely many points we excluded from $\cC$ to get $\widehat{\cC}$, which are algebraic, are defined over $k$. Clearly, there are $f_1, \dots , f_n \in k[T]$ such that $f_j(x_j,\lambda)=0$ for every $j$, identically on the curve.

Let $\cC'$ be the set of points of $\widehat{\cC}$ such that $P_1,\dots ,P_n$ satisfy two independent relations on the specialized curve and
let $\bo{c}_0 \in \cC'$. Since  $\cC$ is defined over $\Qbar$, the $x_j(\bo{c}_0)$ and $\lambda(\bo{c}_0)$ must be algebraic, unless the $P_j$ are identically 
linearly dependent, which we excluded by hypothesis. Then by Silverman's Specialization Theorem \cite{Sil83} (see also Appendix C of \cite{Zannier}) there exists $\gamma_1>0$ such that 
\begin{equation}\label{boundedheight}
h(\lambda(\bo{c}_0))\leq \gamma_1.
\end{equation}

We see now a few consequences of this bound.
If $\delta>0$ is a small real number, let us call
\begin{equation*} 
\Lambda_{\delta}=\lg t \in \C: |t| \leq \frac{1}{\delta}, |t-\lambda(\bo{c})|\geq  \delta \mbox{ for all $\bo{c} \in \cC\setminus \widehat{\cC}$} \rg .
\end{equation*}

\begin{lemma}\label{lemconj}
There is a positive $\delta$ such that there are at least $\frac{1}{2}[k(\lambda(\bo{c}_0)):k]$ different $k$-embeddings $\sigma$ of $k(\lambda(\bo{c}_0))$ in $\C$ such 
that $\sigma (\lambda(\bo{c}_0))$ lies in $\Lambda_{\delta}$ for all $\bo{c}_0 \in \cC'$. 
\end{lemma}

\begin{proof}
See Lemma 8.2 of \cite{MZ12}.
\end{proof}

\begin{remark}
We would like to point out that, as suggested by the referee, it might be possible to avoid the restriction to a compact domain and the use of the previous lemma by exploiting the work of Peterzil and Starchenko \cite{PetStar}, who proved that it is possible to define the Weierstrass $\wp$ function globally in the structure $\R_{\text{an,exp}}$.
\end{remark}

\begin{lemma}\label{lemboundheightdeg}
There exist positive constants $\gamma_2$, $\gamma_3$ such that, for every $\bo{c}_0 \in \cC'$ and every $j=1, \dots , n$, we have
$$
\widehat{h}(P_j(\bo{c}_0))\leq \gamma_2,
$$
and the $P_j(\bo{c}_0)$ are defined over some number field $K\supseteq k(\lambda(\bo{c}_0))$ with
$$
[K:\Q]\leq \gamma_3 d_0,
$$
where $d_0=[k(\lambda(\bo{c}_0)):k]$.
\end{lemma}

\begin{proof}

Recall that each $x_j(\bo{c}_0)$ is a root of $f_j(X,\lambda(\bo{c}_0))$. This already implies the second statement. Now, we have $h(P_j(\bo{c}_0)) \leq \gamma_4(h(\lambda(\bo{c}_0))+1)$ and, using the work of Zimmer \cite{Zimmer}, we have $\widehat{h}(P_j(\bo{c}_0))\leq h(P_j(\bo{c}_0)) +\gamma_5(h(\lambda(\bo{c}_0))+1)$. The first claim now follows from (\ref{boundedheight}).
\end{proof}

\section{Proof of Theorem \ref{mainthm}}

We want to show that there are at most finitely many $\bo{c}$ on the curve such that $P_1(\bo{c}),\dots , P_n(\bo{c})$ satisfy two linear independent 
relations on $E_{\lambda(\bo{c})}$. By Northcott's Theorem \cite{Northcott1949} and (\ref{boundedheight}), we only need to bound the degree $d$ of $\lambda(\bo{c})$ over 
$k$. 

Let $\bo{c}_0 \in \cC'$, $\lambda_0=\lambda(\bo{c}_0)$ and $d_0=[k(\lambda(\bo{c}_0)):k]$. First, by Lemma \ref{lemconj}, we can choose $\delta$, independent of $\bo{c}_0 $, such that $\lambda_0$ has at least 
$\frac{1}{2}d_0$ conjugates in $\Lambda_{\delta}$.
Now, since $\Lambda_{\delta}$ is compact, it can be covered by $\gamma_2$ closed discs $D_{\bo{c}_1}, \dots ,D_{\bo{c}_{\gamma_2}} \subseteq \lambda(\widehat{\cC})$, where $D_{\bo{c}_i}$ is centered in $\lambda(\bo{c}_i)$, for some $\bo{c}_i \in \widehat{\cC}$.

We can suppose that the closed disc $D_{\bo{c}_1}$ contains at least $\frac{1}{2\gamma_2}d_0$ conjugates $\lambda_0^\sigma$. 
Now, each such conjugate comes from a $\bo{c}_0^\sigma \in N_{\bo{c}_1}$ and the corresponding points $P_1(\bo{c}_0^\sigma), \dots ,P_n(\bo{c}_0^\sigma)$ 
satisfy the same linear relations. So there are linearly independent $(a_{1},\dots , a_{n}),(b_{1},\dots , b_{n})$ such that
\begin{align}\label{eq}
a_1P_1(\bo{c}_0^\sigma)+\dots +a_nP_n(\bo{c}_0^\sigma)=
b_1P_1(\bo{c}_0^\sigma)+\dots +b_nP_n(\bo{c}_0^\sigma)= O,
\end{align}
on $E_\lambda (\bo{c}_0^\sigma)$.

By Lemma \ref{lemboundheightdeg}, $\widehat{h}(P_j(\bo{c}_0^\sigma))\leq \gamma_3$ and the points are defined over some finite extension of $k(\lambda(\bo{c}_0^{\sigma}))$ of degree at most $\gamma_4 d_0$. Therefore, applying Lemma \ref{lemgenerators} and recalling \eqref{boundedheight}, we can suppose that the $a_j$ and $b_j$ are in absolute value less than or equal to $\gamma_{5} d_0^{\gamma_{6}}$.

Now, recall that, in Section \ref{periods}, on $\lambda(N_{\bo{c}_1})\supseteq D_{\bo{c}_1}$ we defined $f^{\bo{c}_1},g^{\bo{c}_1}$ to be generators of the period lattice $L_{\lambda}$ and the elliptic logarithms $z_1^{\bo{c}_1}, \dots ,z_n^{\bo{c}_1}$ such that
\begin{equation}\label{elllog}
\exp_\lambda (z_j^{\bo{c}_1}(\lambda))=P_j(\lambda),
\end{equation}
on $\lambda(N_{\bo{c}_1})$. We know that $z_1^{\bo{c}_1}, \dots , z_n^{\bo{c}_1}, f^{\bo{c}_1}, g^{\bo{c}_1}$ are holomorphic functions on a neighborhood of $D_{\bo{c}_1}$, with $f^{\bo{c}_1}(\lambda)$ and $g^{\bo{c}_1}(\lambda)$ linearly independent over $\R$ for every $\lambda\in D_{\bo{c}_1}$, and, by Lemma \ref{lemtransc}, $z^{\bo{c}_1}_{1}, \dots ,z^{\bo{c}_1}_{n}$ are algebraically independent over $\C(f^{\bo{c}_1},g^{\bo{c}_1})$ on $D_{\bo{c}_1}$. Therefore, the hypotheses of Proposition \ref{mainprop2} are satisfied.

By \eqref{eq} and \eqref{elllog}, we have that
\begin{align*}
a_1z_1^{\bo{c}_1}(\lambda_0^\sigma)+\dots +a_nz_n^{\bo{c}_1}(\lambda_0^\sigma)\equiv b_1z_1^{\bo{c}_1}(\lambda_0^\sigma)+\dots +b_nz_n^{\bo{c}_1}(\lambda_0^\sigma)\equiv 0 \mod L_{\lambda_0^\sigma}.
\end{align*}
Therefore, there are $a_{n+1},a_{n+2},b_{n+1},b_{n+2} \in \Z$ such that 
\begin{equation*}
\lg
\begin{array}{c}
a_{1}z_{1}^{\bo{c}_1}( \lambda_0^\sigma )+ \dots + a_{n} z_{n}^{\bo{c}_1}(\lambda_0^\sigma)=a_{n+1} f^{\bo{c}_1}(\lambda_0^\sigma)+ a_{n+2}g^{\bo{c}_1}(\lambda_0^\sigma)\\
b_{1}z_{1}^{\bo{c}_1}(\lambda_0^\sigma)+\dots + b_{n} z_{n}^{\bo{c}_1}(\lambda_0^\sigma)=b_{n+1}f^{\bo{c}_1}(\lambda_0^\sigma) + b_{n+2}g^{\bo{c}_1}(\lambda_0^\sigma).
\end{array}
\right.
\end{equation*} 

Thus all $\lambda_0^\sigma\in D_{\bo{c}_1}$ are in $ D_{\bo{c}_1}(\gamma_{5} d_0^{\gamma_{6}})$ (recall the definition of $D(T)$ right above Proposition \ref{mainprop2}).

By Proposition \ref{mainprop2}, we have that $|D_{\bo{c}_1}(\gamma_{5} d_0^{\gamma_{6}})|\ll_\epsilon d_0^{\gamma_6 \epsilon}$.
But by our choice of $D_{\bo{c}_1}$ we had at least $\frac{1}{2\gamma_2}d_0$ points in $D_{\bo{c}_1}(\gamma_{5} d_0^{\gamma_{6}})$. Therefore, if we choose $\epsilon < \frac{1}{\gamma_6}$ we have a contradiction if $d_0$ is large enough.

We just deduced that $d_0$ is bounded and, by (\ref{boundedheight}) and Northcott's Theorem, we have finiteness of the possible values of $\lambda(\bo{c}_0)$, which proves Theorem \ref{mainthm}.

\section*{Acknowledgments}

The authors would like to thank Umberto Zannier for his support, Daniel Bertrand, Philipp Habegger, Lars K\"uhne, Vincenzo Mantova and Harry Schmidt for useful discussions and the referee for very helpful comments that improved this article.

\bibliographystyle{amsalpha}
\bibliography{bibliography.bib}

\end{document}